\documentclass{article}
\newcommand{\dst}{\displaystyle}

\def\be{\begin{equation}}
\def\ee{\end{equation}}
\def\ba{\begin{array}}
\def\ea{\end{array}}
\def\eqa{\begin{eqnarray}}
\def\eqe{\end{eqnarray}}
\def\R{{\mathbb{R}}}

\newtheorem{theorem}{Theorem}
\newtheorem{corollary}{Corollary}
\newtheorem{proposition}{Proposition}
\newtheorem{lemma}{Lemma}
\newenvironment{definition}{\medskip\noindent{\it Definition. }}{ \medskip}
\newenvironment{proof}{\medskip\noindent{\it Proof.}}{\medskip}

\newenvironment{remark}{\medskip\noindent{\it Remark. }}{
\medskip}

\usepackage{graphicx}        
\usepackage{color}        
\usepackage{amsfonts}
\usepackage{amsmath}
\usepackage{amssymb}  

\title{\LARGE \bf
Minimal data rate stabilization of nonlinear systems over
networks with large delays}

\author{
Claudio De Persis \\
Dipartimento di Informatica e Sistemistica\\
Sapienza Universit\`a di Roma\\
Via Ariosto 25\\ 00185 Roma, ITALY \\
{\tt depersis@dis.uniroma1.it}
}

\date{}

\begin{document}

\maketitle

%
%
%
%
%





\begin{abstract}
Control systems over networks with a finite data rate can be
conveniently modeled as hybrid (impulsive) systems. For the class of
nonlinear systems in feedfoward form,  we design a hybrid
controller which guarantees stability, in spite of the measurement noise
due to the quantization, {\em and} of
an arbitrarily large delay which
affects the communication channel. The rate at which feedback
packets are transmitted from the sensors to the actuators is shown
to be arbitrarily close to the infimal one.
\end{abstract}

\section{Introduction} \label{sec:1}
The problem of controlling systems under communication constraints
has attracted much interest in recent years.
In particular,
many papers have investigated how to cope with the finite bandwidth of
the communication channel in the feedback loop.
For the case of linear systems  (cf.~\cite{brockett.liberzon.tac00,elia.mitter.tac00,
petersen.savkin.cdc01,ishii.francis.tac02,fagnani.zampieri.tac03,nair.evans.aut03,tatikonda.mitter.tac04,
cepeda.astolfi.acc04} to cite a few)
the problem has been very well
understood, and an elegant characterization of the minimal data rate
-- that is the minimal rate at which the measured information
must be transmitted to the actuators --
above which stabilization is always possible is available.
Loosely speaking, the result  shows that the minimal data rate
is proportional to the inverse of the product of the
unstable eigenvalues of the dynamic matrix of the system.
Controlling using the minimal data rate is
interesting not only from a theoretical point of view, but also from
a practical one, even in the presence of communication
channels with a large bandwidth. Indeed, having control techniques which
employ a small number of bits to encode the feedback information
implies for instance  that the number of
different tasks which can be simultaneously carried out
is maximized, results in explicit procedures to convert the
analog information provided by the sensors into the digital form
which can be transmitted, and improves the performance
of the system (\cite{lemmon.sun.cdc06}). We refer the reader to
\cite{nair.et.al.ieee07} for an excellent survey on the topic of control
under data rate constraints. \\
The problem for nonlinear systems has been investigated as well
(cf.~\cite{liberzon.aut03,liberzon.hespanha.tac05,
nair.et.al.tac04, depersis.isidori.scl04, kaliora.astolfi.tac04, depersis.tac05,
depersis.ijrnc06}).
In \cite{liberzon.aut03}, the author extends
the results of \cite{brockett.liberzon.tac00} on quantized control
to  nonlinear systems
which are {\em input-to-state} stabilizable. For the same class,
the paper \cite{liberzon.hespanha.tac05} shows that the
approach in \cite{tatikonda.mitter.tac04} can be employed also for
continuous-time nonlinear systems, although in \cite{liberzon.hespanha.tac05}
no attention is paid on
the minimal data rate needed to achieve the result. In fact, if the
requirement on the data rate is not strict, as it is implicitly
assumed in \cite{liberzon.hespanha.tac05}, it is shown in
\cite{depersis.isidori.scl04} that the  results of
\cite{liberzon.hespanha.tac05} actually hold for the much broader
class of {\em stabilizable} systems. This observation is useful also
to address the case in which only output feedback is transmitted
through the channel. This is investigated in
\cite{depersis.ijrnc06} for the class of {\em uniformly completely
observable} systems.
The paper
\cite{nair.et.al.tac04} shows, among the other results, that
a minimal data rate theorem for {\em local} stabilizability of nonlinear
systems can be proven by focusing on linearized system. To the best of our knowledge,
{\em non} local results
for the problem of minimal data rate stabilization of nonlinear
systems are basically missing. Nevertheless, the paper \cite{depersis.tac05} has pointed out
that, if one restricts the attention to the class of nonlinear
{\em feedforward} systems, then it is possible to find the infimal data
rate above which stabilizability is possible. We recall that feedforward
systems represent a very important class of nonlinear
systems, which has received much attention in recent years
(see e.g. \cite{teel.tac96,mazenc.praly.tac96,
jankovic.et.al.tac96,isidori.book99,marconi.isidori.scl00}, to cite a few),
in which many physical systems fall (\cite{isidori.et.al.book03}),  and for which
it is possible to design stabilizing control laws in spite of
saturation on the actuators.
When
{\em no} communication channel
is present in the feedback loop,
a recent paper (\cite{mazenc.et.al.tac04}, see also \cite{mazenc.et.al.tac03}) has shown that
any
feedforward nonlinear system can be stabilized regardless  of
an arbitrarily
large delay affecting the control action.\\ In this contribution, exploiting
the results of \cite{mazenc.et.al.tac04}, we  show
that the minimal data rate theorem of \cite{depersis.tac05}
holds when an arbitrarily large delay affects the channel (in
\cite{depersis.tac05}, instantaneous delivery through the channel
of the feedback packets was assumed).
Note that the communication channel not only introduces
a delay, but also a quantization error and an impulsive behavior, since the
packets of bits containing the feedback information
are sent only at discrete times. Hence, the
methods of \cite{mazenc.et.al.tac04}, which are studied for continuous-time delay
systems, can not be directly used to deal with impulsive delay systems in the presence
of measurement errors. On the other hand, the approach of
\cite{depersis.tac05} must be modified to take into account the
presence of the delay. To do this, we
use different proof techniques than in \cite{depersis.tac05}.
Namely, while in the former paper the arguments were mainly
trajectory-based, here we employ a mix of trajectory- and
Lyapunov-based techniques. As a consequence, our
result requires an appropriate redesign not only of the parameters in the feedback law of
\cite{mazenc.et.al.tac04}, but also of the encoder and the decoder of
\cite{depersis.tac05}. Moreover, differently from \cite{depersis.tac05},
the parameters of the
feedback control law are now explicitly computed  (see Theorem
\ref{main.result} below). See \cite{liberzon.tac06} for another approach to
control
problems in the presence of delays and quantization. \\
In the next section, we present some preliminary notions useful
to formulate the problem.
The main contribution is stated in Section \ref{section.main.result}.
Building on the coordinate transformations of
\cite{teel.scl92,mazenc.et.al.tac04}, we introduce in Section \ref{section.coord}
a form for the
closed-loop system which is convenient for the
analysis discussed in Section \ref{section.analysis}.
In the conclusions, it is emphasized how the proposed solution is
also robust with respect to packet drop-out.
The rest of the section summarizes the
notation adopted throughout the paper.

\medskip

\noindent {\bf Notation.}
Given an integer $1\le i\le \nu$, the vector
$(a_i,\ldots,a_\nu)\in \R^{\nu-i+1}$ will be succinctly denoted by the
corresponding uppercase letter with index $i$, i.e.~$A_i$. For $i=1$, we will
equivalently use the symbol $A_1$ or simply $a$. $I_i$ denotes the $i\times i$
identity matrix. ${\bf 0}_{i\times j}$ (respectively, ${\bf 1}_{i\times j}$)
denotes
an $i\times j$ matrix whose entries are all $0$ (respectively, $1$).
When only one index is present, it is meant that the matrix is a
(row or column) vector.\\
If $x$ is a vector, $|x|$ denotes the
standard Euclidean norm, i.e.~$|x|=\sqrt{x^Tx}$, while $|x|_\infty$
denotes the infinity norm $\max_{1\le i\le n} |x_i|$. The vector
$(x^T\;y^T)^T$ will be simply denoted as $(x,y)$.
$\mathbb{Z}_+$ (respectively, $\mathbb{R}_+$)
is the set of nonnegative integers (real numbers), $\mathbb{R}^n_+$ is the
positive orthant of $\mathbb{R}^n$.
A matrix $M$ is said to be Schur stable if all its eigenvalues are
strictly inside the unit circle. The symbol $||M||$ denotes the induced matrix 2-norm.\\
The symbol ${\rm sgn}(x)$, with $x$
a scalar variable,  denotes the sign  function  which is equal to $1$ if $x>
0$, $0$ if $x=0$,
and equal to $-1$ otherwise.
If $x$ is an $n$-dimensional  vector, then ${\rm sgn}(x)$ is an $n$-dimensional
vector whose $i$th
component is given by ${\rm sgn}(x_i)$. Moreover, ${\rm diag}(x)$
is an $n\times n$ diagonal matrix whose element $(i,i)$ is $x_i$.\\
Given a vector-valued function of time $x(\cdot)\,:\,\R_+\to \R^n$, the symbol
$||x(\cdot)||_\infty$ denotes the supremum norm  $||x(\cdot)||_\infty=\sup_{t\in
\R_+}|x(t)|$.
Moreover, $x(\bar t^+)$ represents the right limit $\lim_{t\to \bar
t^+}x(t)$. In the paper, two time scales are used, one denoted by the variable $t$
in which the delay is $\theta$, and
the other one denoted by $r$, in which the delay is $\tau$.
Depending on the time scale, the following two norms are used:
$||x_t||=\sup_{-2\theta\le \varsigma\le 0}|x(t+\varsigma)|$,
$||x_r||=\sup_{-2\tau\le \sigma\le 0}|x(r+\sigma)|$.
\\
The saturation
function \cite{mazenc.et.al.tac04} $\sigma\,:\,\R\to\R$ is an odd ${\cal C}^1$ function such
that $0\le \sigma'(s)\le 1$ for all $s\in \R$, $\sigma(s)=1$ for all
$s\ge 21/20$, and $\sigma(s)=s$ for all $0\le s\le 19/20$.
Furthermore, $\sigma_i(s)=\varepsilon_i\sigma(s/\varepsilon_i)$,
with $\varepsilon_i$ a positive real number. The functions
$p_i,q_i\,:\,\mathbb{R}^{n-i+1}\to
\mathbb{R}$ are those functions such that
\cite{mazenc.et.al.tac04,teel.scl92}
\be\label{pq}
\ba{c}
p_i(a_i,\ldots,a_n)=
\dst\sum_{j=i}^{n}\dst\frac{(n-i)!}{(n-j)!(j-i)!}a_j\;,\quad
q_i(a_i,\ldots,a_n)=
\dst\sum_{j=i}^{n}(-1)^{i+j}\dst\frac{(n-i)!}{(n-j)!(j-i)!}a_j\;,
\ea
\ee
with
$p_i(q_i(a_i,\ldots,a_n),\ldots,q_n(a_n))=a_i$,
$q_i(p_i(a_i,\ldots,a_n),\ldots,p_n(a_n))=a_i$.

\vspace{-0.5cm}

\section{Preliminaries}
Consider a nonlinear system
in feedforward form \cite{teel.tac96,mazenc.praly.tac96,
jankovic.et.al.tac96,marconi.isidori.scl00}, that
is a system of the form
\be\label{process}
\dot x(t)
=f(x(t),u(t))
:=
\left(\ba{c} x_2(t)+h_1(X_2(t))\\ \ldots\\ x_n(t)+h_{n-1}(X_n(t))\\ u(t)
\ea\right)\;,
\ee
where  $x_i(t)\in \R$, $X_i(t)$ is the vector of state variables $x_{i}(t),
x_{i+1}(t), \ldots, x_n(t)$, $u(t)\in \mathbb{R}$, each function $h_i$ is $C^2$, and
there exists a positive real number $M>0$ such that for all $i=1,2,\ldots,n-1$, if
$|X_{i+1}|_\infty\le 1$, then
\be\label{M}
|h_i(X_{i+1})|\le M |X_{i+1}|^2\;.
\ee
We additionally assume that a bound on the compact set of initial
conditions is available, namely
a vector $\bar\ell\in \R_+^n$ is known for
which
\be\label{a.ic}
|x_i(t_0)|\le \bar\ell_i\;,\quad i=1,2,\ldots,n\;.
\ee
We investigate the problem of stabilizing the system
above, when the measurements of the state variables travel through a
communication channel. There are several ways to model the effect of the channel.
In the present setting, we assume that there exists a sequence of
strictly increasing transmission times $\{t_k\}_{k\in \mathbb{Z}_+}$,
satisfying
\be\label{tts}
T_m\le
t_{k+1}-t_k\le T_M\;,\quad k\in \mathbb{Z}_+
\ee
for some positive and known
constants $T_m, T_M$, at which a packet of $N(t_k)$ bits, encoding the feedback
information, is transmitted. The packet is received at the other end of the channel $\theta$
units of time later, namely at the times $\theta_k:=t_k+\theta$. In problems of control under
communication constraints, it is interesting to characterize how often
the sensed information is transmitted to the actuators. In this
contribution,
as a measure of the data rate employed by the communication
scheme we adopt the {\em average data rate} \cite{tatikonda.mitter.tac04} defined as
\be\label{adr}
R_{av}=\dst{\lim\sup}_{k\to+\infty} \sum_{j=0}^k \dst\frac{N(t_j)}{t_k-t_0}\;,
\ee
where
$\sum_{j=0}^k N(t_j)$ is the total number of bits transmitted during the time interval
$[t_0,t_k]$. An {\em encoder} carries out the conversion of the state
variables into packets of
bits. At each
time $t_k$, the encoder first samples the state vector to obtain $x(t_k)$, and then
determines a vector $y(t_k)$ of symbols which can be transmitted through the
channel. We recall below the encoder which has been proposed in
\cite{depersis.tac05},  inspired by \cite{tatikonda.mitter.tac04,
liberzon.hespanha.tac05}, and then propose a modification to
handle the presence of the delay.
\subsection{Encoder in the delay-free case}
The encoder
in \cite{depersis.tac05} is as follows:
\be\label{encoder.no.delay}\ba{rcll}
\dot \xi(t)  &=&
f(\xi(t), u(t)) & \\
\dot\ell(t) &=& \mathbf{0}_{n} &
t\ne t_k\\[2mm]
\xi(t^+) &=& \xi(t)+g_{\cal E}(x(t),\xi(t),\ell(t)) & \\
\ell(t^+) &=& \Lambda \ell(t) &  t=t_k\\[2mm]
y(t) &=& {\rm sgn}(x(t)-\xi(t)) & t=t_k
\;,
\ea\ee
where $\xi,\ell$ is the encoder state, $y$ is the feedback information transmitted
through the channel, $\Lambda$ is a
Schur stable matrix, and
$g_{\cal E}(x,\xi,\ell)=4^{-1}{\rm diag}
\left[{\rm sgn}\left(
x-\xi
\right)\right]\ell$. The
rationale behind the encoder (\ref{encoder.no.delay}) is easily
explained (we refer the reader to \cite{tatikonda.mitter.tac04,
liberzon.hespanha.tac05, depersis.tac05} for more details).
During the time at which no new feedback information is
encoded, that is for $t\ne t_k$, the encoder tracks the state of the
system. This explains the first equation of
(\ref{encoder.no.delay}), which is a copy of the system (\ref{process}).
The (positive) values which appear in the entries of  $\ell$  are
the lengths of the edges of the
quantization region, which is a cuboid with center $\xi$.
This vector does not change during continuous flow.
At $t=t_k$, when a new feedback information must be encoded and
transmitted, the encoder splits the cuboid into $2^n$
subregions of equal volume, and select the subregion where the state
$x$ lies. The center of this subregion is taken as the new
state $\xi$ of the encoder (see the third equation in
(\ref{encoder.no.delay})), and communicates this choice of the new
state to the decoder through the symbol $y$. Note that  each
component of $y$ takes value in $\{0,\pm 1\}$, therefore $y$ can
be transmitted as a packet of bits of finite length. In
particular, if $\xi_i$ is on the left of $x_i$ then $+1$ is
transmitted, if it is on the right, then $-1$ is transmitted.
Finally, the size of the quantization region is updated, as it is
shown by the fourth equation in (\ref{encoder.no.delay}).\\
The system (\ref{encoder.no.delay})
is an {\em impulsive} system (\cite{bainov.simeonov.book89,
lakshmikantham.et.al.book, nesic.teel.tac04}) and its behavior
is as follows. At $t=t_0$, given an initial condition $\xi(t_0),\ell(t_0)$,
the updates $\xi(t_0^+),\ell(t_0^+)$ of the encoder state
and $y(t^+)$ of the output are obtained. The former update serves as initial
condition for the continuous-time dynamics, and the state $\xi(t),\ell(t)$
is computed over the interval $[t_0, t_1]$. At the endpoint of the
interval, a new update $\xi(t_1^+),\ell(t_1^+)$ is obtained and the
procedure can be iterated an infinite number of times
to compute the solution $\xi(t),\ell(t)$ for all $t$.\\
At the other end of the channel lies a decoder, which receives the
packets $y(t_k)$, and reconstructs  the state of the system. The
decoder is very similar to the encoder. In fact, we have:
\be\label{decoder.solution}\ba{rcll}
\dot \psi(t)&=& f(\psi(t), u(t)) & \\
\dot \nu(t) &=& \mathbf{0}_{n} & t\ne t_k\\[2mm]
\psi(t^+)&=& \psi(t)+ g_{\cal D}(y(t),\nu(t))
& \\
\nu(t^+) &=& \Lambda \nu(t) & t= t_k
%
\ea\ee
with $g_{\cal D}(y,\nu)=4^{-1}
{\rm diag}(y)\nu$. The control law is the well-known nested saturated feedback
controller
\be\label{control}
\ba{rcl}
u(t)=\alpha(\psi(t))&=&-\dst\frac{L}{M\kappa^n}
\sigma_n\left(p_n\left(\kappa^{n-1}\dst\frac{M}{L}
\psi_{n}(t)\right)+
\sigma_{n-1}\left(p_{n-1}\left(\kappa^{n-2}\dst\frac{M}{L}\psi_{n-1}(t),
\right.\right.\right.\\
&&\left.
\kappa^{n-1}\dst\frac{M}{L}
\psi_{n}(t)\right)+
\ldots+\sigma_{i}\left(p_{i}\left(\kappa^{i-1}\dst\frac{M}{L}\psi_{i}(t),\ldots,
\kappa^{n-1}\dst\frac{M}{L}
\psi_{n}(t)\right)\right.
\\[0mm]
&&\left.\left.
+\lambda_{i-1}(t)\right)
\ldots
\right)\;,
\ea
\ee
where
\[\ba{rcl}
\lambda_{i-1}(t) &=&
\sigma_{i-1}\left(p_{i-1}\left(\kappa^{i-2}\dst\frac{M}{L}\psi_{i-1}(t),\ldots,
\kappa^{n-1}\dst\frac{M}{L}
\psi_{n}(t)\right)\right.\\
&& \left.+\ldots+ \sigma_{1}\left(p_{1}\left(\dst\frac{M}{L}\psi_{1}(t),\ldots,
\kappa^{n-1}\dst\frac{M}{L}
\psi_{n}(t)\right)\right)\ldots\right)\;,
\ea\]
the functions $p_i$ are defined in (\ref{pq}),
and
the parameters $\kappa, L$ and the saturation levels $\varepsilon_i$ of
$\sigma_i(r)=\varepsilon_i\sigma(r/\varepsilon_i)$
are to be designed (see Theorem \ref{main.result}). Note that
we assume the decoder and the actuator
to be
{\em co-located}, and hence this control law is feasible.
If the encoder and the decoder agree to set
their initial conditions to the same value,  then it is not hard to
see (\cite{depersis.tac05}) that $\xi(t)=\psi(t)$ and $\ell(t)=\nu(t)$ for all $t$.
Moreover,  one additionally proves that $\xi(t)$
is an asymptotically correct estimate of $x(t)$, and the latter
converges to zero \cite{depersis.tac05}.

\subsection{Encoders for delayed channels}

When a delay affects the channel, both the encoder and the decoder,
as well as the parameters of the controller, must be modified.
As in \cite{mazenc.et.al.tac04,teel.scl92}, we shall adopt a {\em linear} change of coordinates
in which the control system takes a special form convenient for the
analysis. Differently from \cite{depersis.tac05},
this change of coordinates plays a role also in the
encoding/decoding procedure. Denoted by $\Phi$ the nonsingular matrix
which defines the
change of coordinates, and which we specify in detail in Section \ref{section.coord}, the functions
$g_{\cal E}$, $g_{\cal D}$ which appear in the encoder and, respectively,
the decoder are  modified  as
\[
g_{\cal E}(x,\xi,\ell)=(4\Phi)^{-1}{\rm diag}
\left[{\rm sgn}\left(
\Phi(x-\xi)
\right)\right]\ell\;,\quad
g_{\cal D}(y,\nu)=(4\Phi)^{-1}
{\rm diag}(y)\nu\;.
\]
Now, observe that
the decoder does not know the
first state sample
throughout
the interval $[t_0, t_0+\theta]$,
and hence it can not provide any feedback control action. The
control is therefore  set to zero. As the successive samples   $y(t_k)$
are all received at times $\theta_k=t_k+\theta$,
the decoder becomes aware
of the value of $\xi$  $\theta$ units of time later. Hence, the
best one can expect is to reconstruct the value of $\xi(t-\theta)$ (see Lemma \ref{ds}
below),
and to this purpose  the following decoder is proposed:
\be\label{decoder.delay}\ba{rcll}
\dot \psi(t)&=& f(\psi(t), \alpha(\psi(t-\theta))) & \\
\dot \nu(t) &=& \mathbf{0}_{n} & t\ne \theta_k\\[2mm]
\psi(t^+)&=& \psi(t)+ g_{\cal D}(y(t-\theta),\nu(t))
& \\
\nu(t^+) &=& \Lambda \nu(t) & t= \theta_k\\
u(t) &=& \alpha(\psi(t))\;.
%
\ea\ee
Observe that, as $\dot x(t)=f(x(t),\alpha(\psi(t)))$, and bearing in mind that
$\psi(t)$ reconstructs $\xi(t-\theta)$, the state  $x(t)$
is driven by
a control law which is a function of $\xi(t-\theta)$. This
justifies the appearance of $\psi(t-\theta)$, rather than $\psi(t)$,
in the first equation of the decoder above.\\
We also need to modify the encoder. Indeed,
as mentioned in the case with no delay,
for the encoder to work correctly, the control law (\ref{control}), and hence
$\psi(t)$,
must be available to the encoder. To reconstruct this quantity,
the following equations are added to the encoder (\ref{encoder.no.delay}):
\[
\ba{ccll}
\dot \omega(t) &=& f(\omega(t), \alpha(\omega(t-\theta))) & t\ne \theta_k\\
\omega(t^+) &=& \omega(t)+g_{\cal E}(x(t-\theta),\xi(t-\theta),\ell(t-\theta))
& t=\theta_k\;.
\ea\]
The initial conditions of the encoder and decoder are set as
\be\label{initial}
\ba{c}
\omega(t)=\mathbf{0}_n\;,\; t\in [\theta_0-\theta,\theta_0]\;,\quad
\xi(t_0)=\mathbf{0}_n\;,\quad \ell(t_0)=2\Phi\bar \ell\;,\\[2mm]
\psi(t)=\mathbf{0}_n\;,\; t\in [\theta_0-\theta,\theta_0]
\quad \nu(\theta_0)=2\Phi\bar \ell\;,
\ea
\ee
and, finally, the vector $y$ which is transmitted through the
channel takes the expression
\[
y(t) = {\rm sgn}(\Phi(x(t)-\xi(t)))\;.
\]
Overall, the equations which describe the encoder are:
\be\label{encoder.delay}
\ba{ccll}
\dot \omega(t) &=& f(\omega(t), \alpha(\omega(t-\theta))) & t\ne \theta_k\\
\dot \xi(t)  &=& f(\xi(t), \alpha(\omega(t))) &  t\ne t_k \\
\dot\ell(t) &=& \mathbf{0}_{n} &  t\ne t_k\\
\omega(t^+) &=& \omega(t)+g_{\cal E}(x(t -\theta),\xi(t -\theta),\ell(t-\theta))
& t=\theta_k\\
\xi(t^+) &=& \xi(t)+g_{\cal E}(x(t),\xi(t),\ell(t)) & t=t_k \\
\ell(t^+) &=& \Lambda \ell(t) & t=t_k \\[0.0cm]
y(t) &=& {\rm sgn}(\Phi(x(t)-\xi(t))) & t=t_k
\;.
\ea\ee
The following can be easily proven:
\begin{lemma}\label{ds}
In the above setting, we have: (i) $\omega(t)=\psi(t)$ for all $t\ge t_0$,
(ii) $\xi(t-\theta)=\psi(t)$ and $\nu(t-\theta)=\ell(t)$ for all $t\ge
\theta_0$.
\end{lemma}

\begin{proof}
First observe that the solution of (\ref{process}),
(\ref{decoder.delay}), (\ref{encoder.delay})
exists for all $t\ge t_0$. Further, the thesis is
trivially true for $t\in [t_0, \theta_0]$, by (\ref{initial}).
Now,  for all $t\neq \theta_k$,
$\omega(t)-\psi(t)$ satisfies
\be\label{gr}
\ba{rcl}
\dst\frac{d}{dt}[\omega(t)-\psi(t)]&=&
f(\omega(t), \alpha(\omega(t-\theta)))- f(\psi(t), \alpha(\psi(t-\theta)))
\ea\ee
By definition
of $y$,  after the reset
the initial conditions become
\be\label{ab}\ba{rcl}
\omega(\theta_0^+)
&=& \omega(\theta_0)+(4\Phi)^{-1}{\rm diag}
\left[{\rm sgn}(\Phi(x(t_0)-\xi(t_0)))\right]\ell(t_0)\\
&=& \psi(\theta_0)+(4\Phi)^{-1}
y(t_0)\nu(\theta_0)=\psi(\theta_0^+)\\
\ea\ee
As for $t\in [\theta_0, \theta_0+\theta]$  the right-hand side
of (\ref{gr}) is equal to zero, surely we have that $\omega(t)=\psi(t)$ over such time interval
provided that no state reset occurs. Even when a state
reset occurs (this is the case for instance if the delay is larger than the sampling interval)
we can draw the same conclusion. In fact, before the reset takes place, say at
time $\bar t$,
$\omega_n(\bar t)=\psi_{n}(\bar t)$ for what we have just observed. But then,
exactly as in (\ref{ab}), immediately after the reset the two variables
are equal, and they continue to be the same up to time $\theta_0+\theta$
because the growth of their difference is identically
zero and other possible resets produce no effect as before.
Iterating these arguments shows (i). For the
second part of the thesis, observe that, by definition,
$\xi(t_0)=\psi(t_0)$, and that
$\xi(t_k)=\psi(\theta_k)$ implies $\xi(t_k^+)=\psi(\theta_k^+)$ exactly as in (\ref{ab}).
Moreover, at $t\ne \theta_k$,
$d[\xi(t-\theta)-\psi(t)]/dt=0$, for $d\xi(t-\theta)/dt$ depends
on $\alpha(\omega(t-\theta))$, while
$d\psi(t)/dt$ depends on  $\alpha(\psi(t-\theta))$, and we have
already proven that $\omega(t)=\psi(t)$ for all $t\ge t_0$.
The second  part of (ii) is trivially true.
\end{proof}

As anticipated,
the encoder and decoder we
introduced above are such that the internal state of the former is
exactly reconstructed from the internal state of the latter.
Hence,  in the
analysis to come it is enough  to focus on
the equations describing the process and the decoder only.

\vspace{-0.5cm}

\section{Main result}\label{section.main.result}
The problem we tackle in this paper is, given any value of
the delay $\theta$, find the control (\ref{control}) and
the matrices $\Lambda,
\Phi$ in (\ref{encoder.delay}) and
(\ref{decoder.delay}) which
guarantee the state of the entire closed-loop system to converge to the origin.
As recalled in the previous section,
at times $t_k$, the  measured state is sampled, packed into a sequence of $N(t_k)$
bits, and  fed back to the controller. In other words,
the information flows from the sensors to the actuator with
an average rate $R_{av}$ given by (\ref{adr}). In this setting,
it is therefore meaningful to formulate the problem of
stabilizing the system {\em while} transmitting
the minimal amount of feedback information per unit of time,
that is using the minimal average
data rate. The problem can be formally cast as follows:

\begin{definition}
System (\ref{process}) is
semi-globally asymptotically and locally exponentially stabilizable
using
an average data rate arbitrarily close to the infimal one if,
for any $\bar \ell\in \R_+^n$, $\theta>0$, $\hat R>0$, a controller (\ref{control}),
an encoder
(\ref{encoder.delay}), a decoder (\ref{decoder.delay}),
and initial conditions (\ref{a.ic}),
(\ref{initial}),
exist such that for the closed-loop system  with state
$X:=(x, \omega,\xi,\ell, \psi, \nu)$:\\
(i)
There exist a compact set $C$ containing the origin,
and  $T> 0$, such that $X(t)\in C$
for all
$t\ge T$;\\
(ii) For all $t\ge T$, for some positive real numbers $k, \delta$,
$|X(t)|\le k||X_{T}||
\exp(-\delta
(t-T))$;\\
(iii) $R_{av}<\hat R$.
\end{definition}

\begin{remark}
In the definition above we are slightly abusing the terminology,
since for the sake of simplicity no requirement on Lyapunov (simple) stability
is added. Nevertheless, it is not difficult to  prove that
the origin is a stable equilibrium point,
in the sense that, for each $\varepsilon>0$, there exists $\delta_\varepsilon>0$
such that, $||x_{\theta_0}||< \delta_\varepsilon$ implies $|x(t)|<\varepsilon$
for all $t\ge \theta_0$.
\end{remark}

\begin{remark}
Item (iii) explains what is meant by stabilizability using
an average data rate arbitrarily close to the infimal one.
As a matter of fact, (iii) requires that the average data rate can
be made arbitrarily close to zero, which
of course is the infimal data rate. It is ``infimal" rather than ``minimal",
because we could never stabilize an open-loop unstable system such as
(\ref{process}) with a zero data rate (no feedback).
\end{remark}

Compared with the papers
\cite{teel.tac96,mazenc.praly.tac96,
jankovic.et.al.tac96,marconi.isidori.scl00}, concerned with the
stabilization problem of nonlinear feedforward systems, the novelty here is
due to the presence of impulses, quantization noise which affects the
measurements and delays which affect the control action (on the other hand, we
neglect parametric uncertainty, considered in \cite{marconi.isidori.scl00}). In
\cite{teel.esaim96}, it was shown robustness with respect to
measurement errors for non-impulsive systems with no delay. In
\cite{mazenc.et.al.tac04}, the input is delayed, but neither impulses
nor measurement errors are present. Impulses and measurement errors
are considered in \cite{depersis.tac05}, where the minimal data rate
stabilization problem is solved,
but instantaneous delivery
of the packets is assumed.

\noindent
We state the main result of the paper:
\begin{theorem}\label{main.result}
If in (\ref{control})
\be\label{varepsilon}
\ba{l}
1=30\varepsilon_n=30^2\varepsilon_{n-1}=\ldots=30^n\varepsilon_{1}\\
\kappa \ge \theta \max\left\{6\cdot (30)^{n+1}n(n+1),
8n^2(8n(1+n^2)^{n-1}+1)^2\right\}\\
0<L\le
\min
\left\{
\dst\frac{\kappa}{20\cdot 30^{n}n^4 (n!)^3},
\dst\frac{\kappa}{8(1+n^2)^{n-1}\sqrt{n}n^3 (n!)^3},
 \dst\frac{M\kappa}{(n+1)!}, M
\right\}\;,
\ea\ee
then system (\ref{process}) is
semi-globally asymptotically and locally exponentially stabilizable with
an average data rate arbitrarily close to the infimal one.
\end{theorem}

\vspace{-.25cm}

\begin{remark}
The parameters of the nested saturated controller are very similar
to those in \cite{mazenc.et.al.tac04}, although $\kappa$ and
$L$, $\varepsilon_i$'s are, respectively, larger
and smaller than the corresponding values of
\cite{mazenc.et.al.tac04}, to accommodate the presence of the quantization error.
Moreover, in the proof, the values of the matrices $\Lambda,\Phi$ are
explicitly determined as well.
\end{remark}

\vspace{-.5cm}

\begin{remark}
This result can be viewed as a nonlinear generalization of the well-known data
rate theorem for linear systems. Indeed,
the linearization of the feedforward system
at the origin is
a chain of integrators, for which the minimal data rate theorem for
linear {\em continuous-time} systems states that stabilizability is possible using
an average data rate arbitrarily close to zero.
\end{remark}

\vspace{-.5cm}

\section{Change of coordinates}\label{section.coord}
Building on the coordinate transformations in
\cite{mazenc.et.al.tac04,teel.scl92}, we put the system composed of
the process and the decoder
in a special form (see (\ref{sei}) below). Before doing this, we recall that for feedforward systems
encoders, decoders and controllers are designed in a recursive way
\cite{teel.scl92,teel.tac96,mazenc.praly.tac96,jankovic.et.al.tac96,mazenc.et.al.tac04,depersis.tac05}.
In particular, at each step $i=1,2,\ldots,n$, one focuses
on the last $n-i+1$ equations of system
(\ref{process}),
design the last $n-i+1$ equations of the encoder and the decoder, the first $i$ terms
of the nested saturated controller, and then proceed to the next step, where
the last $n-i$ equations of (\ref{process}) are considered. To this
end, it is useful to introduce additional notation to denote these
subsystems. In particular, for $i=1,2,\ldots, n$, we denote the
last $n-i+1$ equations of
(\ref{process}) by
\be\label{un}
\ba{l}
\dot X_{i}(t)=
H_i(X_{i+1}(t),u(t))=
\left(\ba{c}
x_{i+1}(t)+h_i(X_{i+1}(t))\\
\ldots\\
x_n(t)+h_{n-1}(X_n(t))\\
u(t)
\ea\right)\;,
\ea\ee
with $u(t)=\alpha(\psi(t))$,
while for the last $n-i+1$ equations of the decoder (\ref{decoder.delay}) we adopt
the notation
\be\label{duu}\ba{rcll}
\dot\Psi_{i}(t)&=&H_i(\Psi_{i+1}(t),u(t-\theta))  \\
\dot N_{i}(t)&=& \mathbf{0}_{n-i+1} & t\ne \theta_k\;,\\[2mm]
\Psi_{i}(t^+)&=&\Psi_{i}(t)+(4\Phi_{i})^{-1}
{\rm diag}(Y_i(t-\theta))N_{i}(t)   \\[2mm]
N_{i}(t)&=& \Lambda_i N_{i}(t) &  t=\theta_k\;,
\ea\ee
where $N_i$ denotes the components from $i$ to $n$ of $\nu$.
Moreover, for given positive constants $L\le M$,
$\kappa\ge 1$, with $M$ defined in (\ref{M}), we
define the {\em non singular positive} matrices\footnote{The matrix
$\Phi_1$ will be simply referred to as $\Phi$.} $\Phi_i$ as:
\be\label{nsm}\ba{l}
\Phi_i X_i:=
\left[\ba{c}
p_i\left(\dst\frac{M}{L}\kappa^{i-1}x_i, \ldots,
\dst\frac{M}{L}\kappa^{n-1}x_n\right)\\
\ldots\\
p_n\left(\dst\frac{M}{L}\kappa^{n-1}x_n\right)
\ea\right]\;,\\
i=1,\ldots, n\;,
\ea\ee
where the functions
$p_i$ are those defined in (\ref{pq}).
Finally, let us also introduce the change of {\em time scale}
\be\label{cts}
t=\kappa r\;,
\ee
and the {\em
input}
coordinate change
\be\label{icc}
v(r)=\kappa p_n\left(\dst\frac{M}{L}\kappa^{n-1}u(\kappa
r)\right)\;.
\ee
Then we have:
\begin{lemma}\label{coc}
Let $i\in \{1,2,\ldots,n\}$ and
\be\label{tau}
 \tau:=\theta/\kappa\;,\quad  r_k:=t_k/\kappa\;,\quad
\rho_{k}:=\theta_{k}/\kappa\;.
\ee
The change of coordinates
(\ref{cts}), (\ref{icc}), and
\be\label{stcc}
\ba{rcl}
Z_i(r)&=&\Phi_i X_i(\kappa r)\\
E_i(r)&=&\Phi_{i}(\Psi_{i}(\kappa r)-X_i(\kappa(r-\tau)))\;,\\
P_{i}(r)&=&N_{i}(\kappa r)
\ea\ee
transforms system
(\ref{un})-(\ref{duu}) into
\be\label{sei}
\ba{rcll}
\dot{Z}_i(r)&=& \Gamma_i Z_i(r)+
\mathbf{1}_{n-i+1} v(r)+
\varphi_i(Z_{i+1}(r))\\
\dot{E}_i(r) &=&
\Gamma_i E_{i}(r)+
\varphi_i(E_{i+1}(r)+Z_{i+1}(r-\tau))\\
&& -\varphi_i(Z_{i+1}(r-\tau))
\\
\dot P_i(r) &=& \mathbf{0}_{n-i+1} & r\ne \rho_k\\[2mm]
Z_i(r^+)&=& Z_i(r)\\
E_i(r^+)&=& E_i(r)
+4^{-1}
{\rm diag}({\rm sgn}(-E_i(r)))
P_{i}(r)\\
P_i(r^+) &=& \Lambda_i P_i(r) & r= \rho_k
\;,
\ea\ee
where
\[
\Gamma_i:=\left[\ba{cccccc}
0 & 1 & 1 & \ldots & 1 & 1\\
0 & 0 & 1 & \ldots & 1 & 1\\
\vdots &\vdots &\vdots &\vdots &\vdots &\vdots \\
0 & 0 & 0 & \ldots & 0 & 1\\
0 & 0 & 0 & \ldots & 0 & 0
\ea\right]\;,
\quad
\varphi_i(Z_{i+1}):=\left[\ba{c}
f_i(Z_{i+1}) \\ f_{i+1}(Z_{i+2}) \\ \ldots \\ f_{n-1}(Z_{n}) \\
0
\ea\right]\;,
\]
where
\be\label{sette}
|f_i(Z_{i+1})|\le P |Z_{i+1}|^2\;,\quad P=n^3(n!)^3 L \kappa^{-1}\;,
\ee
provided that $|Z_{i+1}|_\infty\le (M\kappa)/(L(n+1)!)$.
\end{lemma}

\begin{proof}
The proof is easily derived from analogous arguments in
\cite{mazenc.et.al.tac04,teel.scl92}.
However, for the convenience of the Readers, we are adding it in the Appendix.
\end{proof}

In the new coordinates (\ref{cts})-(\ref{icc}), (\ref{stcc}),
the
controller (\ref{control}) takes the form
\be\label{v.r}\ba{rcl}
v(r)&=& -\sigma_n(e_n(r)+z_n(r-\tau)+\sigma_{n-1}(e_{n-1}(r)+z_{n-1}(r-\tau)+\\
&&
\ldots+\sigma_i(e_i(r)+z_i(r-\tau)+\hat\lambda_{i-1}(r))\ldots))\;,
\ea\ee
with
$\hat\lambda_{i-1}(r)=\sigma_{i-1}(e_{i-1}(r)+z_{i-1}(r-\tau)+\ldots+\sigma_1(e_1(r)+
z_1(r-\tau))\ldots)$.

\section{Analysis}\label{section.analysis}
In the previous sections,
we have introduced the encoder, the decoder and the controller.
In  this section, in order to show the stability property,
we carry out a  step-by-step analysis, where at each step
$i$, we consider the subsystem (\ref{sei})  in closed-loop
with (\ref{v.r}).
We first recall two lemmas which are at the basis of the
iterative construction. The first one, which, in a different
form, was basically given in
\cite{depersis.tac05}, shows that the decoder asymptotically tracks
the state of the process under a boundedness assumption.

\begin{lemma}\label{mcs}
Suppose (\ref{a.ic}) is true. If for some $i=1,2,\ldots, n$
there exists a positive real number $\bar{\rm Z}_{i+1}$
such that \footnote{The conditions are void for $i=n$.}
\[
||Z_{i+1}(\cdot)||_\infty\le \bar{\rm Z}_{i+1}\;,
\]
and, for all
$r\ge \rho_0$,
\[
|e_j(r)|\le p_{j}(r)/2\;,\quad j=i+1,i+2,\ldots, n\;,
\]
with\footnote{In the statement, the
continuous dynamics of the impulsive systems are trivial --
the associated vector fields are identically zero -- and hence omitted.}
\[
P_{i+1}(\rho^+)=\Lambda_{i+1}P_{i+1}(\rho)\quad
\rho= \rho_k\;,
\]
and $\Lambda_{i+1}$ a Schur stable matrix, then for all
$r\ge \rho_0$,
\[
|e_i(r)|\le p_{i}(r)/2\;,
\]
with
$p_{i}(r^+)= p_{i}(r)/2$, for $r=\rho_k$,
if $i=n$, and
\be\label{ai}\ba{rcll}
\left[\ba{c}
p_{i}(r^+)\\
P_{i+1}(r^+)
\ea\right]
&=&
\left[\ba{cc}
1/2 & \ast \\ {\bf 0}_{n-i} & \Lambda_{i+1}
\ea\right]
\left[\ba{c}
p_{i}(r)\\
P_{i+1}(r)
\ea\right] & r= \rho_k \;,
\ea\ee
if $i\in\{1,2,\ldots, n-1\}$, where
$\ast$ is a $1\times (n-i)$ row vector depending on $\bar Z_{i+1}$, $\bar \ell$, and $T_M$.
\end{lemma}

\begin{proof}
See \cite{depersis.tac05}. For the convenience of the Readers, we are adding
the proof in the Appendix.
\end{proof}


The following observation is useful later on.
As a consequence of the mean value theorem, it is not difficult to realize
(see also \cite{depersis.tac05})
that, if
$||z(\cdot)||_\infty\le Z$, for some $Z>0$, then $e$ and $p$ in
(\ref{sei}) (with $i=1$) obey the equations\footnote{Again, we adopt the
symbol $\Lambda$ rather than $\Lambda_1$.}
\be\label{ep}
\ba{rcll}
\dot{e}(r)&=& A(r)e(r)\\
\dot{p}(r)&=& {\bf 0}_{n} & r\ne\rho_k\\
e(r^+)&=& e(r)
+4^{-1}
{\rm diag}[{\rm sgn}(-e(r))]
p(r)  \\
p(r^+)&=&\Lambda p(r) & r=\rho_k\;,
\ea
\ee
with
\be\label{ep0}
A(r):=
\left[\ba{cccccc}
0 & a_{12}(r) & a_{13}(r) & \ldots & a_{1\,n-1}(r) & a_{1n}(r)\\
0 & 0 & a_{23}(r) & \ldots & a_{2\,n-1}(r) & a_{2n}(r)\\
\vdots &\vdots &\vdots &\vdots &\vdots &\vdots \\
0 & 0 & 0 & \ldots & 0 & a_{n-1\,n}(r)\\
0 & 0 & 0 & \ldots & 0 & 0
\ea\right]\;,
\ee
and where the off-diagonal components of $A$, rather than as
functions of $(r,e(r),z(r-\tau))$, are viewed as bounded (unknown)
functions of
$r$,  whose absolute value can be assumed without
loss of generality to be upper bounded
by a positive constant depending on $Z$, $\bar \ell$ and
$T_M$. Now,
concisely rewrite the  system (\ref{ep})
as (\cite{bainov.simeonov.book89})
\be\label{epsilon}\ba{rcll}
\dot\epsilon(r)&=& B(r)\epsilon(r) & r\ne \rho_k\\
\epsilon(r^+)&=& g_k(\epsilon(r)) & r= \rho_k\;,
\ea\ee
with $\epsilon=(e,p)$, $|g_k(\epsilon)|\ge |\epsilon|/2$. The
following straightforward result shows that for the system above an
exponential Lyapunov function exists (this fact was not pointed out
in \cite{depersis.tac05}, where the proofs were not Lyapunov-based):
\begin{corollary}\label{cor}
There exists a function $V(r,\epsilon)=V(r,e,p)\,:\,\R_+\times
\R^{n}\times \R^{n}\to \R_+$ such that, for all $r\in \R_+$ and for all
$\epsilon=(e,p) \in
\R^{n}\times \R^{n}$ for which $|e|\le |p|/2$, satisfies
\[
\ba{rl}
c_1|\epsilon|^2\le V(r,\epsilon)\le c_2|\epsilon|^2 & \\[2mm]
\dst\frac{\partial V}{\partial r}+ \dst\frac{\partial V}{\partial
\epsilon}B(r)\epsilon\le -c_3|\epsilon|^2 & r\ne \rho_k\\[2mm]
V(r^+,g_k(\epsilon))\le V(r,\epsilon) & r= \rho_k\\[2mm]
\left|
\dst\frac{\partial V(r,\epsilon)}{\partial \epsilon}
\right|\le c_4|\epsilon|\;, &
\ea
\]
for some positive constants $c_i$, $i=1,\ldots, 4$.
\end{corollary}

\begin{proof}
By Lemma \ref{mcs}, for any $r$,
$|e(r)|^2+|p(r)|^2\le 5 |p(\rho_k^+)|^2/4$,
where $k$ is the index such that $r\in [\rho_k, \rho_{k+1})$. Since
the matrix
$\Lambda$ is Schur stable, there exist positive constants $\mu$ and
$\hat\lambda<1$ such that
\[\ba{rcl}
|p(\rho_k^+)|&\le& \mu \hat\lambda^{k+1} |p(\rho_0)|
\le   \mu \exp(-\lambda(r-\rho_0)) |p(\rho_0)|
\ea\]
with $\lambda=|\ln \hat\lambda|/T_m$. We conclude that
\[
\left|\epsilon(r)
\right|\le \sqrt{5/4}\mu  \exp(-\lambda(r-\rho_0)) |p(\rho_0)|
\le \sqrt{5/4}\mu  \exp(-\lambda(r-\rho_0))
\left|\epsilon(\rho_0)
\right|
\;,
\]
that is the solution is exponentially stable. Then,
we can apply a standard converse Lyapunov theorem for
impulsive systems, such as
Theorem 15.2 in \cite{bainov.simeonov.book89}, and infer the thesis.
\end{proof}


\noindent The next statement, based on  Lemma 10 in \cite{mazenc.et.al.tac04},
shows that a controller exists which
guarantees the boundedness of the state variables, a property
required in the latter result. For the proof, observe
that the arguments of \cite{mazenc.et.al.tac04}
hold  even in the presence of a ``measurement"
disturbance $e$ induced by the quantization, which can be possibly
large during the transient but it  is decaying to
zero asymptotically.

\begin{lemma}\label{l3}
Consider the system
\[
\dot{Z}(r)= -\varepsilon\sigma\left[
\dst\frac{1}{\varepsilon}(Z(r-\tau)+e(r)+\lambda(r))
\right]+\mu(r)
\]
where $Z\in \R$, $\varepsilon$ is a positive real number, and
additionally:
\begin{itemize}
\item $\lambda(\cdot)$ and $\mu(\cdot)$ are continuous functions for
which positive real numbers $\lambda_\ast$ and $\mu_\ast$ exist
such that, respectively,
$|\lambda(r)|\le \lambda_\ast$, $|\mu(r)|\le \mu_\ast$, for all
$r\ge r_0$.
\item $e(\cdot)$ is a piecewise-continuous function for which a positive
time $r_\ast$ and a positive number $e_\ast$ exist such that
$|e(r)|\le e_\ast$, for all $r\ge
r_\ast$.
\end{itemize}
If
\[\ba{l}
\tau \in
\left(0, \dst\frac{1}{12}\right]\;,\quad
\lambda_\ast \in
\left(
0, \dst\frac{\varepsilon}{30}
\right]\;,
e_\ast \in
\left(
0, \dst\frac{\varepsilon}{30}
\right]\;,
\quad \mu_\ast \in
\left(
0, \dst\frac{\varepsilon}{30}
\right]\;,
\ea\]
then there exist positive real numbers $Z_\ast$ and $R\ge 0$
such that $||Z(\cdot)||_\infty \le Z_\ast$, and for all $r\ge R$,
\[
|Z(r)|\le 4(\lambda_\ast+\mu_\ast+e_\ast)\;.
\]
\end{lemma}

\begin{proof}
See \cite{mazenc.et.al.tac04}.
\end{proof}
%

To illustrate the iterative analysis in a concise manner, the
following is very useful (cf.~the analogous inductive
hypothesis in \cite{mazenc.et.al.tac04}):

\medskip

\noindent {\it Inductive Hypothesis} There exists $\bar{\rm Z}_i>0$
such that
$||Z_i(\cdot)||\le \bar{\rm Z}_i$. Moreover, for each $j=i,i+1, \ldots,n$,
$|e_j(r)|\le p_{j}(r)/2$, for all
$r\ge \rho_0$, and there exists $R_i>\tau$ such that
for all
$r\ge R_i$,
\[\ba{l}
|z_j(r)|\le \dst\frac{1}{4}\varepsilon_j\;,\quad
|e_j(r)|\le \dst\frac{1}{2n}\cdot \dst\frac{1}{80^{j-i+2}} \varepsilon_{j}
\;.
\ea\]

\medskip

\noindent {\it Initial step} ($i=n$) The initial step is trivially true, provided that
$\tau\le 1/12$, and $\varepsilon_n=1/30$.

\medskip

\noindent {\it Inductive step} The inductive step is summarized in
the following result:

\begin{lemma}\label{isl}
Let
\be\label{ef}
P \le P_m\le [20\cdot(30)^n
n]^{-1}\;,\quad
\tau  \le\tau_m\le  [6\cdot 30^{n+1}n(n+1)]^{-1}\;.
\ee
If the inductive hypothesis is true for some $i\in \{2,\ldots,n\}$,
then it is also true for $i-1$.
\end{lemma}

\begin{proof}
The proof only slightly differs from an analogous one in
\cite{mazenc.et.al.tac04} in that we have to deal with the
presence of the quantization errors. The arguments are as follows.
If  the induction hypothesis is satisfied
for some $i\in \{2,\ldots,n\}$, then the control law
$v(r)$ in (\ref{v.r})
becomes equal to
\[\ba{l}
-w_{n}(r)-\ldots -w_{i}(r)-\sigma_{i-1}(w_{i-1}(r)+\hat
\lambda_{i-2}(r))=\\
-\dst\sum_{j=i}^n [e_j(r)+z_j(r-\tau)]-\sigma_{i-1}(w_{i-1}(r)+\hat
\lambda_{i-2}(r))\;,
\ea\]
provided that $|e_j(r)+z_j(r-\tau)|\le 11 \varepsilon_j/12$ for each
$j=i,i+1,\ldots, n$, which is actually the case by hypothesis
for all $r\ge R_i+\tau$.
As a consequence, the components from $i-1$ to $n$ of the system in the variable $z$
writes as
\be\label{wa}
\hspace{-0.25cm}
\ba{rcl}
\dot{z}_{i-1}(r) &=& -\sigma_{i-1}(z_{i-1}(r-\tau)+e_{i-1}(r)+\hat
\lambda_{i-2}(r))+[z_i(r)-z_{i}(r-\tau)+\ldots+z_n(r)\\
&& -z_{n}(r-\tau)]-[e_{i}(r)+\ldots+e_n(r)]+f_{i-1}(Z_i(r))\\[0mm]
\dot{z}_{\ell}(r) &=& \sum_{j=\ell+1}^n z_j(r)-\sum_{j=i}^n e_j(r)-
\sum_{j=i}^n z_j(r-\tau)-\\
&& \sigma_{i-1}(z_{i-1}(r-\tau)+e_{i-1}(r)+\hat
\lambda_{i-2}(r))+f_\ell(Z_{l+1}(r))\hfill \ell=i,\ldots, n-1
\\[2mm]
\dot{z}_{n}(r) &=& -\sum_{j=i}^n e_j(r)-
\sum_{j=i}^n z_j(r-\tau) -\sigma_{i-1}(z_{i-1}(r-\tau)+e_{i-1}(r)+\hat \lambda_{i-2}(r))
\ea\ee
As it is noted in \cite{mazenc.et.al.tac04}, from
the expression of $\dot{z}_\ell(r)$, with
$\ell=i,i+1,\ldots,n$, one proves that
$|\dot{z}_\ell(r)|\le
(n+1)/4+P_m n/16$,
for $|f_\ell(Z_{\ell+1}(r))|\le P_m n/16$.
Hence, for each $\ell=i,i+1,\ldots,n$,
\[
|z_{\ell}(r)-z_{\ell}(r-\tau)|\le |\dst\int_{r-\tau}^{r}
\dot{z}_\ell(\sigma)d\sigma|
\le \tau\left(\dst\frac{n+1}{4}+\dst\frac{1}{16}P_m n\right)\;,
\]
so that
\[
|[z_i(r)-z_{i}(r-\tau)+\ldots+z_n(r)-z_{n}(r-\tau)]+f_{i-1}(Z_i(r))|
\le n\tau\left(\dst\frac{n+1}{4}+\dst\frac{1}{16}P_m n\right)+\dst\frac{1}{16}P_m
n\;.
\]
The following upper bounds easily follow from (\ref{ef}),
\[
\dst\frac{1}{16}P_m n \le \dst\frac{1}{320}\varepsilon_{i-1}\;,\quad
\tau n\dst\frac{n+1}{4}\;,\;
\tau n\cdot \dst\frac{1}{16}P_m n \le \dst\frac{1}{2}\cdot
\dst\frac{1}{320}\varepsilon_{i-1}\;.
\]
By the inductive hypothesis, we additionally have
$|e_{i}(r)+\ldots+e_n(r)|\le \varepsilon_{i-1}/160$.
The bounds above show that in the $\dot{z}_{i-1}(r)$ equation we have
\[
|[z_i(r)-z_{i}(r-\tau)+\ldots+z_n(r)-z_{n}(r-\tau)]-[e_{i}(r)+\ldots+e_n(r)]+f_{i-1}(Z_i(r))|
\le \varepsilon_{i-1}/80\;.
\]
Furthermore, as a consequence of  Lemma \ref{mcs} and the inductive hypothesis,
$e_{i-1}(r)$ satisfies $|e_{i-1}(r)|\le
p_{{i-1}}(r)/2$, and in finite time becomes smaller
than $\varepsilon_{i-2}/(n\cdot 160)\le \varepsilon_{i-1}/30$.
Hence, it is possible to apply Lemma \ref{l3}, and infer that
$||Z_{i-1}(\cdot)||_\infty \le
Z_{i-1}$ and
there exists a time $R_{i-1}$ such that $|z_{i-1}(r)|\le \varepsilon_{i-1}/4$ for all $r\ge
R_{i-1}$.
\end{proof}

\noindent Applying this lemma repeatedly, one
concludes that, after a finite time, the state
converges to the linear operation region for all the saturation
functions, it remains therein for all the subsequent times, and
hence
the closed-loop system starts evolving according
to the equations
(recall (\ref{epsilon}))
\be\label{ep2}
\ba{rcll}
\dot z(r) &=& \Gamma z(r)+\Delta z(r-\tau)+\Delta e(r)+\varphi(z(r)) &\\
\dot{\epsilon}(r)&=& B(r)\epsilon(r) & r\ne\rho_k\\[0mm]
z(r^+) &=& z(r) & \\
\epsilon(r^+)&=& g_k(\epsilon(r)) & r=\rho_k\;,
\ea
\ee
where:\\
(i) $\Gamma , \Delta $ are matrices for which there exist
$q=(1+n^2)^{n-1}$, $a=n$, and $Q=Q^T>0$ such that
\[
(\Gamma +\Delta )^TQ+Q(\Gamma +\Delta )\le -I\;,
\]
with $||Q||\le q$ and $||\Gamma ||,||\Delta ||\le a$;\\
(ii)
There exists $\gamma>0$ such that $\varphi(z(r)):=[\,f_1(Z_2(r))\;\ldots\;f_{n-1}(Z_{n}(r))\;0\,]^T$
satisfies
\[
|\varphi(z)|\le \gamma
|z|\;;
\]
(iii) $\epsilon=(e,p)$.

In \cite{mazenc.et.al.tac04} the authors investigate the
stability property of
\[
\dot z(r) = \Gamma z(r)+\Delta z(r-\tau)+\varphi(z(r))\;,
\]
that is the first component of system (\ref{ep2}), with $e=0$ and
no impulses. In the present case, $e$ is due to the quantization
noise and drives the $z$-subsystem. The ``driver" subsystem is
described by
the $\epsilon$ equations of (\ref{ep2}). Hence, we have to study
the stability of a {\em cascade} system with impulses.
The Corollary \ref{cor} points out that there exists an exponential
Lyapunov function for the $\epsilon$ subsystem of (\ref{ep2}).
Based on
this function, one can build  a Lyapunov-Krasowskii functional to
show that the origin is exponentially stable for the  entire cascade impulsive system
(\ref{ep2}), thus
extending Lemma 11 in \cite{mazenc.et.al.tac04} in the following way:
\begin{lemma}\label{eci}
Consider system (\ref{ep2}), for which the conditions (i)-(iii) hold. If
\[
\gamma\le \dst\frac{1}{8q}
\quad \mbox{and}\quad \tau\le
\dst\frac{1}{8a^2(8aq+1)^2}\;,
\]
then, for all $r\ge \rho_0$, for some positive real numbers $\mu,
\delta$, we have
\[
|(z(r),\epsilon(r))|\le \mu||(z,\epsilon)_{\rho_0}||\exp(-\delta
(r-\rho_0))\;.
\]
\end{lemma}

\begin{proof}
Following \cite{mazenc.et.al.tac04}, we notice first that
$V_1(z)=z^TQz$ satisfies, by the requirement on $\gamma$,
\[\ba{rcl}
\dot V_1&\le&
-\dst\frac{1}{2}|z|^2+
\dst\frac{a^2(8aq+1)^2}{4}\tau\dst\int_{r-2\tau}^r |z(s)|^2
ds
+ 8q^2a^4\tau \dst\int_{r-\tau}^r |\epsilon(s)|^2ds
+
8q^2a^2|\epsilon|^2
\;.
\ea\]
The condition on $\tau$ additionally gives
\[\ba{rcl}
\dot V_1&\le&
-\dst\frac{1}{2}|z|^2+
\dst\frac{1}{32}\dst\int_{r-2\tau}^r |z(s)|^2
ds
+ \dst\int_{r-\tau}^r |\epsilon(s)|^2ds
+
8q^2a^2|\epsilon|^2
\;.
\ea\]
The Lyapunov-Krasowskii functional
\[
V_2(r,z,\epsilon)=V_1(r,z)+\dst\frac{1}{16}\dst\int_{r-2\tau}^{r}\dst\int_{s}^{r}
|z(\ell)|^2 d\ell\,d s
+2\dst\int_{r-2\tau}^{r}\dst\int_{s}^{r}
|\epsilon(\ell)|^2 d\ell\,d s\;,
\]
satisfies for $r\ne \rho_k$
\[\ba{rcl}
\dot V_2&\le&
-\dst\frac{1}{2}|z|^2+
\dst\frac{1}{32}\dst\int_{r-2\tau}^r |z(s)|^2
ds
+ \dst\int_{r-2\tau}^r |\epsilon(s)|^2ds
+
8q^2a^2|\epsilon|^2\\[2mm]
&& -\dst\frac{1}{16}\dst\int_{r-2\tau}^{r}|z(\ell)|^2 d\ell
+\dst\frac{1}{8}\tau |z|^2-2\dst\int_{r-2\tau}^{r}|\epsilon(s)|^2 d\ell
+4\tau |\epsilon|^2
\;,
\ea\]
that is
\[\ba{rcl}
\dot V_2 &\le&
-\dst\frac{1}{4}|z|^2-
\dst\frac{1}{32}\dst\int_{r-2\tau}^r |z(s)|^2
ds
- \dst\int_{r-2\tau}^r |\epsilon(s)|^2ds
+8(q^2a^2+1)|\epsilon|^2
\;,
\ea\]
since $\tau\le 2$. Now, recalling that $|e(r)|\le |p(r)|/2$ for
all $r$ as a consequence of Lemma \ref{mcs}, and choosing the function
$V_3(r,z,\epsilon)=V_2(r,z,\epsilon)+\lambda V(r,\epsilon)$ (see also
\cite{mazenc.bliman.tac06} for a similar choice),
with  $\lambda \ge [1+8(q^2a^2+1)]/c_3$,
it is readily proven that
\be\label{lkfd}\ba{rcl}
\dot V_3 &\le&
-\dst\frac{1}{4}|z|^2-|\epsilon|^2 -
\dst\frac{1}{32}\dst\int_{r-2\tau}^r |z(s)|^2
ds
- \dst\int_{r-2\tau}^r |\epsilon(s)|^2ds
\;,
\ea\ee
for all $r\ne \rho_k$.
The inequality
$V_3(r^+,z,g_k(\epsilon))\le  V_3(r,z,\epsilon)$, for $r= \rho_k$,
also descends straightforwardly
(cf.~\cite{liu.mcm04}). Hence, we can conclude that the impulsive delay system
is exponentially
stable. Indeed,
$a_1 |(z,\epsilon)|^2\le V_3(r,z,\epsilon)\le a_2
||(z,\epsilon)_r||^2$,
with $a_1=\lambda_m(Q)+\lambda c_1$, $a_2=\lambda_M(Q)+3\tau^2/4+\lambda c_2$.
Moreover,
\[
\dot V_3\le
- a_3\left[
\lambda_M(Q)|z|^2 + \lambda c_2 |\epsilon|^2+\dst\frac{\tau}{8}
\dst\int_{r-2\tau}^r |z(s)|^2
ds+2\tau \dst\int_{r-2\tau}^r |\epsilon(s)|^2ds
\right]
\le -a_3
 V_3\;,
\]
with
$a_3^{-1}=\max\left\{4\lambda_M(Q), \lambda c_2,
4\tau\right\}$.
Over the interval $[\rho_k, \rho_{k+1})$, we have
(\cite{halanay.book})
\[
V_3(\rho_{k+1}^+,z,g_{k+1}(\epsilon))\le V_3(\rho_{k+1},z,\epsilon)\le
\exp(-a_3 (\rho_{k+1}-\rho_k)) V_3(\rho_{k}^+,z,g_k(\epsilon))\;.
\]
A simple iterative arguments yields that for all $r\ge \rho_0$,
\[
a_1 |(z(r),\epsilon(r))|^2\le V_3(r,z,\epsilon)\le \exp(-a_3 (r-\rho_0))
V_3(\rho_0,z(\rho_0),\epsilon(\rho_0))\le a_2 ||(z,\epsilon)_{\rho_0}||^2  \exp(-a_3 (r-\rho_0))\;,
\]
that is
$|(z(r),\epsilon(r))|\le \sqrt{a_2/a_1}||(z,\epsilon)_{\rho_0}||\exp(-a_3
(r-\rho_0)/2)$,
which shows exponential stability of the impulsive delay system.
\end{proof}

\noindent We can now state the following stability result for the
system (\ref{sei}), (\ref{v.r}):
\begin{proposition}\label{p1}
If the saturation levels $\varepsilon_i$ are as in (\ref{varepsilon}),
\be\label{taup0}
L\le \min\{M,\dst\frac{M \kappa}{(n+1)!}\}
\ee
and
\be\label{taup}
\ba{rcl}
0\le \tau\le \tau_m &=& [\max\left\{6\cdot (30)^{n+1}n(n+1),
8n^2(8n(1+n^2)^{n-1}+1)^2\right\}]^{-1}\\[2mm]
0\le P\le P_m &=& [\max\{20\cdot 30^{n}n,
8(1+n^2)^{n-1}\sqrt{n}\}]^{-1}
\;,
\ea\ee
then
the closed-loop system (\ref{sei}), with $i=1$, and
$|e_i(\rho_0)|\le
p_{i}(\rho_0)/2$ for all $i=1,2,\ldots,n$,
is semi-globally asymptotically
and locally exponentially stabilized by the controller (\ref{v.r}).
\end{proposition}

\begin{proof}
Bearing in mind (\ref{sette})
and that $\varepsilon_i<1$, for $i=1,2,\ldots, n$, and $(M\kappa)/(L(n+1)!)\ge 1$, then $\gamma$
in (ii) after (\ref{ep2}) is equal to $\sqrt{n}P$, and the
condition $P\le [8(1+n^2)^{n-1}\sqrt{n}]^{-1}$ in
(\ref{taup}) actually implies $\gamma\le 1/(8q)$. Analogously, one
can check that the remaining requirements on $\tau$ and $P$ in (\ref{taup})
imply that all the conditions in Lemma \ref{isl}
and \ref{eci} are true. These lemma
allow us to infer the thesis. Indeed, by Lemma \ref{isl},
the state after a finite time
starts evolving in the linear operation region for all the saturation
functions. Hence, from a certain time on, the state
exponentially converges to zero, again by Lemma \ref{eci}.
\end{proof}

\medskip

\noindent {\it Proof of Theorem \ref{main.result}:}
The proof of the main result of the paper simply amounts to rephrase
the proposition above in the original coordinates. Details
are straightforward and we omit them. Only observe that the
parameter $\kappa$ is chosen so as to satisfy the requirement on
$\tau$ in (\ref{taup}). Once $\kappa$ is fixed, $L$ is chosen to satisfy
 (\ref{taup0}) and the second inequality in
(\ref{taup}). The
equations of
encoder and the decoder are introduced in (\ref{encoder.delay}) and,
respectively, (\ref{decoder.delay}). The matrices $\Lambda$ and
$\Phi$ appearing in (\ref{encoder.delay}) and,
respectively, (\ref{decoder.delay}) are designed in Lemma
\ref{mcs} and, respectively, defined in (\ref{nsm}).\\
As far as
the minimality of the data rate is concerned, we proceed as in
\cite{depersis.tac05}. Observe that, by definition,
$R_{av}=2n/T_m$, and  to guarantee $R_{av}<\hat R$, we must have
$T_m>
2n/\hat R$. Since $\hat R$ can be arbitrarily close to zero,
$T_m$ can be arbitrarily large. Nevertheless, the stability result
shown above holds for any value of $T_m$, for $T_m$ (and hence $T_M\ge T_m$)
affects
the off-diagonal entries of $A(r)$ and $\Lambda$ only, and the exponential stability
of the $\epsilon$ equations (and therefore of system (\ref{ep2}))
remains true, as it is evident from Lemma \ref{mcs} (see also \cite{depersis.tac05}).
Hence, no matter what $\hat R$ is, we can always stabilize the system with $R_{av}<\hat R$.

\section{Conclusion}
We have shown that minimal data rate stabilization of nonlinear
systems is possible even when the communication channel is affected
by an arbitrarily large transmission delay.
The system has been
modeled as the feedback interconnection of a couple of impulsive nonlinear control systems with the
delay affecting the feedback loop.
In suitable coordinates, the closed-loop system turns out to be
described by a cascade of impulsive delay nonlinear control systems,
and semi-global asymptotic plus  local exponential
stability has been shown. The proof relies, among other things, on
the design of a Lyapunov-Krasowskii functional for an appropriate cascade impulsive
time-delay system. If the encoder is
endowed with a device able to detect abrupt changes in the rate of growth
of $x_n$, or if a dedicated channel is available to inform the
encoder about the transmission delays,
then it is not difficult to derive the same kind of
stability result for the case when the
delays are time-varying and upper-bounded by $\theta$. Similarly, by
adjusting $T_M$ in (\ref{tts}), it is possible to show that the
solution proposed in this paper is also robust with respect to
packet drop-outs, provided that the maximal number of consecutive
drop-outs is available.
The same kind of approach appears to be suitable for other problems
of control over communication channel with finite data rate,
delays and packet drop-out.

\bibliography{encoding_delayed_channel}

\begin{center}
APPENDIX
\end{center}

\noindent This appendix is added for the convenience of the Readers.

\medskip

\begin{center}
Proof of Lemma \ref{coc}
\end{center}

It is shown in \cite{mazenc.et.al.tac04} that,
(\ref{cts}), (\ref{icc}) and
$Z_i(r)=\Phi_i X_i(\kappa r)$
transforms
(\ref{un}) into
\be\label{ssei}
\ba{rcl}
\dot{Z}_i(r)&=&F_i(Z_{i+1}(r), v(r))\\[2mm]
&:=&
\left[\ba{c}
\sum_{j=i+1}^{n}z_j(r)+v(r)+f_i(Z_{i+1}(r))\\
\sum_{j=i+2}^{n}z_j(r)+v(r)+f_{i+1}(Z_{i+2}(r))\\
\vdots\\
v(r)
\ea\right]\;.
\ea\ee
Clearly,
the equation (\ref{ssei}) is equal to the first equation of (\ref{sei}).
Bearing in mind (\ref{ssei}), and  by differentiating $E_i$ defined in (\ref{stcc}), we obtain
\be\label{de}\ba{rcl}
\dot{E}_i(r)&=&F_i(E_{i+1}(r)+Z_{i+1}(r-\tau),v(r-\tau))-F_i(Z_{i+1}(r-\tau),v(r-\tau))
\\
&=&
\Gamma_i (E_{i}(r)+Z_{i}(r-\tau))+ \mathbf{1}_{n-i+1} v(r-\tau)+
\varphi_i(E_{i+1}(r)+Z_{i+1}(r-\tau))
\\
&&
-\Gamma_i (Z_{i}(r-\tau))- \mathbf{1}_{n-i+1} v(r-\tau)-
\varphi_i(Z_{i+1}(r-\tau))
\\
&=& \Gamma_i E_i(r)+
\varphi_i(E_{i+1}(r)+Z_{i+1}(r-\tau))-\varphi_i(Z_{i+1}(r-\tau))
\ea\ee
for $r\ne \rho_k$, while for $r=
\rho_k$,  we have:
\be\label{r0}\ba{rcl}
E_i(\rho_k^+)&=& \Phi_i(\Psi_{i}(\kappa \rho_k^+)-X_i(\kappa(\rho_k-\tau)^+))\\
&=& \Phi_i(\Psi_{i}(\theta_k^+)-X_i(t_k^+))\\
&=& \Phi_i(\Psi_{i}(\theta_k)
+(4\Phi_{i})^{-1}
{\rm diag}(Y_i(t_k^+))
N_{i}(\theta_k) -X_i(t_k))\\
&=& \Phi_i(\Psi_{i}(\theta_k)-X_i(t_k))
+ 4^{-1}
{\rm diag}({\rm sgn}(\Phi_i[X_i(t_k)-\Xi_i(t_k)]))
N_{i}(\theta_k)\\[0mm]
&=& \Phi_i(\Psi_{i}(\theta_k)-X_i(t_k))
+ 4^{-1}
{\rm diag}({\rm sgn}(\Phi_i[X_i(t_k)-\Psi_{i}(\theta_k)]))
N_{i}(\theta_k)\;,
\ea\ee
where
the last equality descends from (ii) in Lemma
\ref{ds}, and implies
\be\label{r}
E_i(\rho_k^+)= E_i(\rho_k)
+4^{-1}
{\rm diag}({\rm sgn}(-E_i(\rho_k)))
N_{i}(\theta_k)\;.
\ee
The thesis then follows if we observe that the variable $P_i$
defined in (\ref{stcc}) satisfies
\be\label{p}\ba{rcll}
\dot P_i(r) &=& \mathbf{0}_{n-i+1} & r\ne \rho_k\\
P_i(r^+) &=& \Lambda_i P_i(r) & r= \rho_k\;.
\ea\ee

\begin{center}
Proof of Lemma \ref{mcs}
\end{center}

Recall first the equations in (\ref{sei}) which define
the evolution of $E_i$.
Furthermore, by (\ref{initial}), the definition of $\Phi$, and
(\ref{a.ic}), $|e_j(\rho_0)|\le
p_{j}(\rho_0)/2$ for $j=i,i+1,\ldots, n$.
For $i=n$, as $|e_n(\rho_0)|\le
p_{n}(\rho_0)/2$, it is immediately seen that
\[
|e_n(\rho_0^+)|=| e_n(\rho_0)+4^{-1}
{\rm sgn}(-e_n(\rho_0))
p_{n}(\rho_0)|\le 4^{-1}p_{n}(\rho_0)
\]
which proves that $|e_n(\rho_0^+)|\le p_{n}(\rho_0^+)/2$,
provided that  $\Lambda_n=1/2$. As $\dot{e}_n(r)=0$,
then $|e_n(r)|\le p_{n}(\rho_0^+)/2$ for $r\in [\rho_0,\rho_1)$. As
$\dot{p}_{n}(r)=0$, also $|e_n(\rho_1)|\le p_{n}(\rho_1)/2$,
and iterative arguments prove that $|e_n(r)|\le p_{n}(\rho_{k}^+)/2$
on each interval $[\rho_k,\rho_{k+1})$.
Notice that the single trivial eigenvalue of $\Lambda_n$ is strictly
less than the unity. The equation for $e_i$ in (\ref{sei}) writes as:
\[\ba{rcl}
\dot{e}_i(r) &=& {\bf 1}_{n-i}
E_{i+1}(r)+\varphi_i(E_{i+1}(r)+Z_{i+1}(r-\tau))-\varphi_i(Z_{i+1}(r-\tau))\\
&=& \left({\bf 1}_{n-i}
+
\left[\dst\frac{\partial \varphi_i(y_{i+1})}{\partial y_{i+1}}\right]_{\alpha(r)E_{i+1}(r)+Z_{i+1}(r-\tau)}
\right)E_{i+1}(r)
\;,
\ea\]
with $\alpha(r)\in [0,1]$ for all $r$.
As both $E_{i+1}$ and $Z_{i+1}$ are bounded, it is not hard to see
\cite{depersis.tac05} that there exists a positive real number $F_i$ depending on
$\bar{\rm Z}_{i+1}$ and $\bar \ell$,  such that, for $r\in[\rho_k, \rho_{k+1})$,
\[
e_i(r)\le e_i(\rho_k^+)+F_i(\rho_{k+1}-\rho_k)\dst\sum_{j=i+1}^{n}
p_{j}(\rho_{k}^+)/2\;,
\]
with $|e_i(\rho_0^+)|\le p_{i}(\rho_0)/4$. By iteration, the
thesis is inferred  provided that
\[\ba{rcl}
p_{i}(\rho_{k}^+)&=&\dst\frac{1}{2}p_{i}(\rho_{k})
+F_iT_M{\bf 1}_{n-i}\Lambda_{i+1}P_{i+1}(\rho_{k})\\
&\ge&
\dst\frac{1}{2}p_{i}(\rho_{k})
+F_i(\rho_{k+1}-\rho_k)\dst\sum_{j=i+1}^{n} p_{j}(\rho_{k}^+)\;.
\ea\]
Note that, by the
definition of $p_{i}(\rho_{k}^+)$ above, $P_{i}(\rho_{k}^+)=
\Lambda_i P_{i}(\rho_{k})$, with $\Lambda_i$ the matrix
in (\ref{ai}),
that shows $\Lambda_i$ to be a Schur stable
matrix provided that so is
$\Lambda_{i+1}$.

\end{document}